\documentclass[a4paper,5pt]{article}
\usepackage{varwidth}
\usepackage[english]{babel}
\usepackage{amsfonts,amssymb,amsmath}
\usepackage[utf8]{inputenc}
\usepackage[T1]{fontenc}
\usepackage{lmodern}
\usepackage{graphicx}
\usepackage{fancyvrb}
\usepackage{listings}
\usepackage{hyperref}
\usepackage{amsthm}
\usepackage[nottoc]{tocbibind}
\usepackage{biblatex}
\usepackage{authblk}
\usepackage{fullpage}
\providecommand{\keywords}[1]
{
  \small	
  \textbf{\textit{Keywords---}} #1
}

\begin{document}
\newtheorem{theorem}{Theorem}[section]
\newtheorem{corollary}[theorem]{Corollary}
\newtheorem{lemma}[theorem]{Lemma}
\newtheorem{remark}[theorem]{Remark}
\newtheorem{definition}[theorem]{Definition}
\newtheorem{example}[theorem]{Example}
\newtheorem{conjecture}[theorem]{conjecture}

\title{Explicit geometric construction of Ramsey graphs}

\author[1]{Matija Kocbek\thanks{Faculty of Mathematics and Physics, University of Ljubljana, Slovenia. email: mkocbek1@gmail.com}}
\maketitle

\begin{abstract}
We present an explicit geometric construction of a large parametrized family of graphs with no $k$-cliques and with bounded independence number, generalizing the triangle-free Ramsey graphs of Codenotti, Pudlák, and Resta and revisiting the previous generalization by Kostochka, Pudlák and Rödl. Within this framework, 
we provide a new combinatorial proof of the upper bound on the independence number for their constructions, offering additional insight into the structure of its independent sets.
For our generalized family, we establish lower bounds on the independence number and identify necessary constraints on the parameters under which these graphs could yield improved constructive asymptotic lower bounds on $R(s,t)$, with particular emphasis on $R(3,t)$.
As a byproduct, we provide a linear-time approximation algorithm for finding the largest independent set within this parametrized family. For a substantial subfamily, this algorithm achieves a $\frac{1}{2}$-approximation ratio.
Special attention is given to the case of triangle-free graphs: we describe explicit constructions on $n$ vertices with independence number $O(n^{\frac{2}{3}})$, giving a constructive asymptotic lower bound of $\Omega(t^{\frac{3}{2}})$ for the Ramsey numbers $R(3,t)$, matching the best known constructive bound.
\end{abstract}


\section{Introduction}

A subgraph of some graph is a clique if any two vertices in the subgraph are adjacent. A subgraph is an independent set if no two vertices in it are adjacent. Ramsey theorem \cite{Ramsey1930} in its special case states the following:

    \begin{theorem}
        For any natural $s, t \geq 3$ there is a natural $n$ such that any graph with at least $n$ vertices contains a clique of size $s$ or an independent set of size $t$. We call the smallest such $n$ Ramsey number $R(s, t)$.
    \end{theorem}

Precise determination of Ramsey numbers is an extremely difficult problem, so we often try to understand asymptotic behaviour of these numbers. The only Ramsey numbers whose precise asymptotics we know are Ramsey numbers $R(3, t)$. In 1980 Ajtai, Koml{\'o}s and Szemer{\'e}di \cite{Ajtai1980} proved the asymptotic upper bound on the Ramsey numbers $R(3, t)$ and in 1995 Kim \cite{Kim95} proved non-constructively that the upper bound is tight:

    \begin{theorem}
        $R(3, t) = \Theta\left(\frac{t^2}{\log{t}}\right).$
    \end{theorem}

There is also interest put into attempts of constructive proofs of the asymptotic lower bounds on $R(3, t)$. To this day the best constructive lower bound is given by $R(3, t) = \Omega(t^{\frac{3}{2}})$. The first article which presented an explicit construction of a family of triangle-free graphs with $n$ vertices and independence number of order $O(n^{\frac{2}{3}})$ was \cite{Alon94}. Later, a new construction of such graphs was presented in 
\cite{CodenottiPR00}. 

In 2010 Kostochka, Pudlák and Rödl \cite{Kostochka2010} presented a generalization of graphs in \cite{CodenottiPR00}. They focused on a few specific graphs without cliques of size $4$, $5$ and $6$ respectively which they defined similarly to the triangle-free graphs in \cite{CodenottiPR00}, and with the help of the former graphs they achieved new best constructive asymptotic lower bounds on $R(4, t), R(5, t)$ and $R(6, t)$ which haven't been improved to this day. They were also the first to present 
a combinatorial proof of the upper bound on the independence number of the graphs in \cite{CodenottiPR00}, as the original article presented only an algebraic proof.

In this article we present a generalization of these graphs which is isomorphic to a subset of graphs described in \cite{Kostochka2010}. With the help of this generalization we give a new combinatorial proof of the upper bound on their independence number which suggests how constructions of graphs with smaller independence number could be obtained. We first define an incidence structure that generalizes the notions of $t$-wise balanced designs and packings, which to our knowledge hasn't been explicitly 
named before but is in principle the main subject of the Zarankiewicz problem. We present inequalities connecting the number of points, blocks and incidence pairs in this structure. Based on this incidence relation we describe a large parametrized family of graphs 
without $k$-cliques and using the previously mentioned inequalities we determine lower bounds on the independence number of these graphs and necessary conditions for parameters under which we could obtain graphs that might provide us new and better constructive asymptotic lower bounds on Ramsey numbers $R(s, t)$ for any $s, t \geq 3$, while mostly focusing on $R(3, t)$. As a byproduct, we provide a linear-time approximation algorithm for finding the largest independent set within this parametrized family, which is a $\frac{1}{2}$-approximation algorithm for a significant subfamily.

This parametrized family of graphs yields plenty of previously undescribed triangle-free graphs whose independence number is of order $O(n^k)$ where $n$ is the number of vertices and $0 < k < 1$. We will explicitly state some 
of these graphs, among which will be two new constructions of triangle-free graphs with $n$ vertices and independence number of order $O(n^{\frac{2}{3}})$. One of these two families is the first explicit construction of a triangle-free family of graphs with $n$ vertices and such independence number which exists for any $n$, as both \cite{Alon94} and \cite{CodenottiPR00} presented constructions only for $n$ with certain properties.

\section{Geometric construction and basic properties}

\subsection{Definition and upper bound on the independence number}

Let's first define the central object we will be working with.

    \begin{definition}
        Let $m, \lambda \in \mathbb{N}$, let $\mathcal{A}$ be any set whose elements we call points and let $\mathcal{B} \subseteq \mathcal{P}(A)$ be a set of subsets of $A$ (we call the elements of $\mathcal{B}$ blocks) for which the following conditions hold:
            \begin{enumerate}
                \item Any $m$ distinct points are contained together in at most $\lambda$ blocks.
                \item Every point is contained in at least one block.
                \item Every block is non-empty.
            \end{enumerate}
        Then we say $\mathcal{B}$ is a weakly $m$-wise balanced design over $\mathcal{A}$ with parameter $\lambda$ or an $m$-wise balanced packing over $\mathcal{A}$ with parameter $\lambda$.
    \end{definition}

    \begin{definition}
        Let $m, \lambda \in \mathbb{N}$ and let $\mathcal{B}$ be a weakly $m$-wise balanced design over some set $\mathcal{A}$ with parameter $\lambda$. Let $<$ be any strict linear ordering of the set $\mathcal{A}$. We then define a graph $\Gamma_{\mathcal{B}}^{m,\lambda}$ in such way, 
        that we take take all incidence pairs for vertices and two vertices $(x, B_1)$ and $(y, B_2)$ are adjacent if $x < y$, $B_1$ and $B_2$ are different and $x \in B_2$.
    \end{definition}

Our notation neglects the fact that the above defined graph is also dependent on the set $\mathcal{A}$ and its linear ordering $<$. This is for simplicity of writing but we keep this fact in mind. Let's explore properties of this graph.

    \begin{theorem}
        Let $\mathcal{B}$ be a weakly $m$-wise balanced design over $\mathcal{A}$ with parameter $\lambda$. Then $\Gamma_{\mathcal{B}}^{m,\lambda}$ does not contain a clique of size $m + \lambda$.
    \end{theorem}

    \begin{proof}
        We prove by contradiction. Assume we have a clique $(x_1, B_1), (x_2, B_2), \ldots, (x_{m+\lambda}, B_{m+\lambda})$. Then all points and blocks are pairwise different. Without loss of generality we have $x_1 < x_2 < \ldots < x_{m+\lambda}$. Then it follows by construction
        that $x_1, \ldots x_m \in B_{m+i}$ for all $i \in \{0, \ldots, \lambda\}$. Because $m$ distinct points are contained in at most $\lambda$ distinct blocks, this is a contradiction and so $\Gamma_{\mathcal{B}}^{m, \lambda}$ really doesn't contain an clique of size $m + \lambda$.
    \end{proof}

The following upper bound on the independence number of graphs $\Gamma_{\mathcal{B}}^{m,\lambda}$ has already been proven in \cite{Kostochka2010} where they showed that the bound holds for a class of isomorphic graphs. However, here we provide a different proof by constructing an injective map from an independent set to the union of sets of points and blocks. This proof gives us more information
about the structure of independent sets and reduces the problem of finding an upper bound on size of an independent set to bounding the size of a smaller, different set. We will demonstrate specifically why this proof has additional value in the next subsection.

    \begin{theorem}
        Let $\mathcal{B}$ be a weakly $m$-wise balanced design over $\mathcal{A}$ with parameter $\lambda$. Then $\alpha(\Gamma_{\mathcal{B}}^{m, \lambda}) \leq |\mathcal{A}| + |\mathcal{B}|$.
        \label{upper_bound}
    \end{theorem}

    \begin{proof}
        Let $H \subseteq V(\Gamma_{\mathcal{B}}^{m, \lambda})$ be any independent set of vertices. For every $B \in \mathcal{B}$ denote $B_H$ as the set of points which appear in pair with $B$ in $H$. Then we can write $H$ as a disjoint union of two sets.

        The first set $H_1$ contains all incidence pairs $(x, B) \in H$ for which $x$ is maximal in $B_H$. Then obviously every block appears in $H_1$ at most once, so we have $|H_1| \leq |\mathcal{B}|$.
        
        The other set $H_2$ contains all incidence pairs $(x, B) \in H$ for which $x$ isn't maximal in $B_H$. Notice that $x$ cannot appear in $H_2$ with any other block different from $B$. Let's prove this by contradiction. Assume $x$ appears in pair with some other block $B'$.
        Because $x$ isn't maximal in $B_H$, there is a pair $(y, B) \in H$ for which $x < y$. But this means that $(x, B')$ and $(y, B)$ are adjacent which contradicts independence of $H$. 
        This means that every point appears in $H_2$ at most once. Then of course $|H_2| \leq |\mathcal{A}|$.

        Because $H = H_1 \cup H_2$, we have $|H| \leq |\mathcal{A}| + |\mathcal{B}|$ and so $\alpha(\Gamma_{\mathcal{B}}^{m, \lambda}) \leq |\mathcal{A}| + |\mathcal{B}|$.
    \end{proof}

\subsection{Lower bounds on the independence number}

    \begin{theorem}
        Let $\mathcal{B}$ be a weakly $m$-wise balanced design over $\mathcal{A}$ with parameter $\lambda$. Then $|\mathcal{B}| \leq \alpha(\Gamma_{\mathcal{B}}^{m, \lambda})$.
        \label{lower_bound_1}
    \end{theorem}

    \begin{proof}
        We will prove that every graph $\Gamma_{\mathcal{B}}^{m, \lambda}$ contains an independent set of size $|\mathcal{B}|$. We prove by induction on $n = |\mathcal{B}|$.

        For $n = 0$, the statement is obvious and base of induction fulfilled. Let $n \geq 1$ and suppose the statement is true for all smaller natural numbers. Let $\mathcal{B}$ of power $n$ be a weakly $m$-wise balanced design over $\mathcal{A}$ with parameter $\lambda$. Let $x$ be the smallest point in $\mathcal{A}$ and $\mathcal{C}$ the set of all blocks which contain $x$. Let $\mathcal{I}_1$ be the set of all incidence pairs which contain $x$ and blocks in $\mathcal{C}$.
        This set is obviously independent because all incidence pairs contain the same point in them. Let furthermore $\mathcal{D} = \mathcal{B} \setminus \mathcal{C}$. Then by construction of $\mathcal{C}$, $x$ isn't included in any blocks in $\mathcal{D}$. Let $y \in \mathcal{A} \setminus \{x\}$ and $B_2 \in \mathcal{D}$ such that $y \in B_2$ and let $(x, B_1) \in \mathcal{I}_1$. 
        Because of minimality of $x$, incidence pairs $(x, B_1)$ and $(y, B_2)$ can be adjacent only if $x \in B_2$, which is impossible because $x$ isn't contained in any block in $\mathcal{D}$. Thus $\mathcal{I}_1$ doesn't share an edge with any incidence pair with points in $\mathcal{A}$ and blocks in $\mathcal{D}$ and is different from every such pair. Let $\mathcal{A}' \subseteq \mathcal{A}$ be the
        set of points from $\mathcal{A}$ which are contained in some block in $\mathcal{D}$. Then $\mathcal{D}$ is a weakly $m$-wise balanced design over $\mathcal{A}'$ with parameter $\lambda$ and $\Gamma_{\mathcal{D}}^{m, \lambda}$ is isomorphic to the subgraph $H$ of $\Gamma_{\mathcal{B}}^{m, \lambda}$, which is induced on vertices which have points in $\mathcal{A}'$ and blocks in $\mathcal{D}$. We have proven that this subgraph is independent and disjoint from $\mathcal{I}_1$. 
        Because $x$ is contained in at least one block, $\mathcal{C}$ is a non-empty set, and so $|\mathcal{D}| < n$. Then $\Gamma_{\mathcal{D}}^{m, \lambda}$ and with it $H$ have an independent set $\mathcal{I}_2$ of size $|\mathcal{D}|$. Then $\mathcal{I} = \mathcal{I}_1 \cup \mathcal{I}_2$ is an independent set in $\Gamma_{\mathcal{B}}^{m, \lambda}$ of size $|\mathcal{C}| + |\mathcal{D}| = |\mathcal{B}|$ which completes the proof.
    \end{proof}

As we can see from \ref{lower_bound_1}, we can always find an independent set in $\Gamma_{\mathcal{B}}^{m, \lambda}$ for which the set $H_1$ from the proof of \ref{upper_bound} contains exactly $|\mathcal{B}|$ elements. So the upper bound $|H_1| \leq |\mathcal{B}|$ demonstrated in that proof is tight. However, we can also see that for some graphs $\Gamma_{\mathcal{B}}^{m, \lambda}$, there is no such independent set for which $|H_2| = |\mathcal{A}|$. Take for example $\mathcal{A} = \{1, \ldots, n\}$
for $n \geq 5$ and let $\mathcal{B} = \{\{1, 2, 3\}, \{1, 4, 5, \ldots, n\}\}$. This is a weakly $2$-wise balanced design with parameter $\lambda = 1$ and one can easily see that $|H_2| \leq \alpha(\Gamma_{\mathcal{B}}^{2, 1}) = n - 1 = |\mathcal{A}| - 1$. Thus the upper bound $|H_2| \leq |\mathcal{A}|$ is not tight for some graphs $\Gamma_{\mathcal{B}}^{m, \lambda}$. Achieving a tighter upper bound for $|H_2|$ for some subfamilies of graphs $\Gamma_{\mathcal{B}}^{m, \lambda}$ could result in new better constructive lower bounds for Ramsey numbers.

The proof of \ref{lower_bound_1} gives us an algorithm with which we can efficiently find an independent set of size $|\mathcal{B}|$ in $\Gamma_{\mathcal{B}}^{m, \lambda}$. We start with an empty independent set. Then we take the smallest point and add all incidence pairs in which it appears to the independent set. From the set of all points and blocks we remove this point and all the blocks which contain it and again take the smallest 
of the rest of the points and we add all the incidence pairs in which it appears with non-removed blocks to the independent set. We continue the procedure until we run out of blocks. If we appropriately store incidences, this algorithm has $O(|\mathcal{A}| + |\mathcal{B}|)$ time complexity. 

Theorem \ref{lower_bound_1} also gives us the following obvious corollary which in its special case tells us that if we have more blocks than points, the independence number is proportional to the sum of the number of points and blocks.

    \begin{corollary}
        \label{corollary_lower_bound1}
        Let $\mathcal{B}$ be a weakly $m$-wise balanced design over $\mathcal{A}$ with parameter $\lambda$. Suppose $|\mathcal{A}|^{\varepsilon} \leq k|\mathcal{B}|$ for some $k > 0$ and some $0 < \varepsilon \leq 1$. Then $\frac{|\mathcal{A}|^{\varepsilon} + |\mathcal{B}|}{k + 1} \leq \alpha(\Gamma_{\mathcal{B}}^{m, \lambda})$.
    \end{corollary}

For many important classes of weakly $m$-wise balanced designs it is already known that $|\mathcal{A}| \leq |\mathcal{B}|$, such as for example $t$-designs and pairwise balanced designs under conditions for non-triviality. 
    
Weakly $m$-wise balanced design $\mathcal{B}$ over $\mathcal{A}$ with parameter $\lambda$ is an $m$-design with parameters $(v, k, \lambda)$ if $|\mathcal{A}| = v$, every block contains exactly $k$ points and any $m$ distinct points are contained together in precisely $\lambda$ blocks. Fisher inequality \cite{Fisher40} tells us that if $\mathcal{B}$ is such an $m$-design and $k < v$, then $|\mathcal{A}| \leq |\mathcal{B}|$. This tells us the following.

    \begin{corollary}
        Let $\mathcal{B}$ be an $m$-design with parameters $(v, k, \lambda)$ over $\mathcal{A}$ and let $k < v$. Then $\frac{|\mathcal{A}| + |\mathcal{B}|}{2} \leq \alpha(\Gamma_{\mathcal{B}}^{m, \lambda}) \leq |\mathcal{A}| + |\mathcal{B}|$. 
    \end{corollary}

Likewise, weakly $2$-wise balanced design $\mathcal{B}$ is a pairwise balanced design with parameter $\lambda$ if every block contains at least two points and any two points are contained together in precisely $\lambda$ blocks. Erd\H{o}s-de Bruijn theorem \cite{Bruijn48} tells us that if $\mathcal{B}$ is a pairwise balanced design with paramater $\lambda = 1$ and no block contains all of $\mathcal{A}$, then $|\mathcal{A}| \leq |\mathcal{B}|$. The same property also holds for any $\lambda$.\cite{Stinson2004} This gives us the following corollary.

    \begin{corollary}
        Let $\mathcal{B}$ be a pairwise balanced design over $\mathcal{A}$ with parameter $\lambda$ such that no block contains all of $\mathcal{A}$. Then $\frac{|\mathcal{A}| + |\mathcal{B}|}{2} \leq \alpha(\Gamma_{\mathcal{B}}^{2, \lambda}) \leq |\mathcal{A}| + |\mathcal{B}|$.
    \end{corollary}

In these cases when there are less points than there are blocks, the previously described algorithm is an $\frac{1}{2}$-approximation algorithm.

We also have a different lower bound for the independence number of $\Gamma_{\mathcal{B}}^{m, \lambda}$ which is useful, when we have more points than blocks.

    \begin{theorem}
        \label{lower_bound_2}
        Let $\mathcal{B}$ be a weakly $m$-wise balanced design over $\mathcal{A}$ with parameter $\lambda$. Then $\frac{|\mathcal{A}|}{|\mathcal{B}|} \leq \alpha(\Gamma_{\mathcal{B}}^{m, \lambda})$.
    \end{theorem}

    \begin{proof}
        We will prove that $\Gamma_{\mathcal{B}}^{m, \lambda}$ contains an independent set of size at least $\frac{|\mathcal{A}|}{|\mathcal{B}|}$. Let's prove that some block must contain at least $\frac{|\mathcal{A}|}{|\mathcal{B}|}$ points. Suppose the opposite. Because every point is contained in some block, we have $\bigcup_{B \in \mathcal{B}}B = A$. However, we also have $|\bigcup_{B \in \mathcal{B}}B| < |\mathcal{B}|\frac{|\mathcal{A}|}{|\mathcal{B}|} = |\mathcal{A}|$
        which is a contradiction. So we have a block $B$ which contains at least $\frac{|\mathcal{A}|}{|\mathcal{B}|}$ points. Then the set $\mathcal{I} = \{(x, B); x \in B\}$ is an independent set of size at least $\frac{|\mathcal{A}|}{|\mathcal{B}|}$ which proves the theorem.
    \end{proof}

This, together with \ref{lower_bound_1}, gives us the following corollary.

    \begin{corollary}
        \label{corollary_lower_bound2}
        Let $\mathcal{B}$ be a weakly $m$-wise balanced design over $\mathcal{A}$ with parameter $\lambda$. Suppose $|\mathcal{B}| \leq k|\mathcal{A}|^{1-\varepsilon}$ for some $k > 0$ and $0 < \varepsilon < 1$. Then $\frac{|\mathcal{A}|^{\varepsilon} + k|\mathcal{B}|}{2k} \leq \alpha(\Gamma_{\mathcal{B}}^{m, \lambda})$.
    \end{corollary}

\section{Asymptotic lower bounds on the independence number}

\subsection{Zarankiewicz problem and asymptotic lower bounds for any parameters}

Zarankiewicz problem \cite{Zarankiewicz1951} asks for the number $z(m, n; s, t)$ which is defined as the maximal possible number of entries $1$ in an $m \times n$ matrix, containing only $0$ and $1$ and which doesn't contain an $s \times t$ minor. Equivalently, we can view $z(m, n; s, t)$ as the maximal number of edges in a bipartite graph $G$ with bipartition $U \cup V$ where $|U| = m, |V| = n$ amd $G$ doesn't contain a subgraph $K_{s, t}$ with bipartition $S \cup T$ for which
$|S| = s, |T| = t, S \subseteq U$ and $T \subseteq V$. The famous K\H{o}v\'{a}ri-S\'{o}s-Tur\'{a}n theorem \cite{Kovari1954} gives for constant $s$ and $t$ the upper bound
$$z(m, n; s, t) = O(mn^{1 - \frac{1}{s}} + n).$$
Swapping the roles of $m$ and $n$, $s$ and $t$, we get another upper bound
$$z(m, n; s, t) = O(nm^{1 - \frac{1}{t}} + m).$$
We notice that the incidence matrix of some weakly $m$-wise balanced design $\mathcal{B}$ over $\mathcal{A}$ with parameter $\lambda$ is precisely an $a \times b$ matrix without an $m \times (\lambda + 1)$ minor, where $|\mathcal{A}| = a, |\mathcal{B}| = b$. If we denote the number of incidence pairs with $n$, this give us the inequality
$$n \leq z(a, b; m, \lambda + 1).$$
By the previous asymptotic upper bounds for the Zarankiewicz problem, we also know that there exist functions $f$ and $g$ so that
\begin{align}
    \label{zarankiewicz}
    n &\leq f(m, \lambda)(ab^{1 - \frac{1}{m}} + b)\\
    n &\leq g(m, \lambda)(ba^{1 - \frac{1}{\lambda + 1}} + a).
\end{align}

These inequalities give us the following corollary.

    \begin{corollary}
        Let $\mathcal{B}$ be a weakly $m$-wise balanced design over $\mathcal{A}$ with parameter $\lambda$ and $n$ incidence pairs. Let $\varepsilon > 0$ and $\varepsilon \leq 1$. Then $|\mathcal{A}|^{\varepsilon} + |\mathcal{B}| = \Omega(n^{\frac{m\varepsilon}{m\varepsilon + m - \varepsilon}})$.
    \end{corollary}

    \begin{proof}
        If $|\mathcal{B}| = \Omega(n^{\frac{m\varepsilon}{m\varepsilon + m - \varepsilon}})$, the statement is obvious. Otherwise we have $|\mathcal{B}| = o(n^{\frac{m\varepsilon}{m\varepsilon + m - \varepsilon}})$ and so by \ref{zarankiewicz} $|\mathcal{A}|^{\varepsilon} = \Omega(n^{\frac{m\varepsilon}{m\varepsilon + m - \varepsilon}})$ which proves the corollary.
    \end{proof}

This corollary in turn gives us the following lower bounds for the independence number.

    \begin{corollary}
        \label{corollary1_zarankiewicz}
        Let $\mathcal{B}$ be a weakly $m$-wise balanced design over $\mathcal{A}$ with parameter $\lambda$. Let $|\mathcal{A}|^{\varepsilon} \leq k|\mathcal{B}|$ for some $k > 0$ and some $0 < \varepsilon \leq 1$, i. e. $|\mathcal{A}|^{\varepsilon} = O(|\mathcal{B}|)$. Let $n$ be the number of incidence pairs, i. e. the number of vertices in $\Gamma_{\mathcal{B}}^{m, \lambda}$. Then $\alpha(\Gamma_{\mathcal{B}}^{m, \lambda}) = \Omega(n^{\frac{m\varepsilon}{m\varepsilon + m - \varepsilon}})$.
    \end{corollary}

    \begin{corollary}
        \label{corollary2_zarankiewicz}
        Let $\mathcal{B}$ be a weakly $m$-wise balanced design over $\mathcal{A}$ with parameter $\lambda$. Let $|\mathcal{B}| \leq k|\mathcal{A}|^{1-\varepsilon}$ for some $k > 0$ and some $0 < \varepsilon < 1$, i. e. $|\mathcal{B}| = O(|\mathcal{A}|^{1-\varepsilon})$. Let $n$ be the number of incidence pairs, i. e. the number of vertices in $\Gamma_{\mathcal{B}}^{m, \lambda}$. Then $\alpha(\Gamma_{\mathcal{B}}^{m, \lambda}) = \Omega(n^{\frac{m\varepsilon}{m\varepsilon + m - \varepsilon}})$.
    \end{corollary}

By observing values $\varepsilon \geq \frac{1}{2}$, we see that \ref{corollary1_zarankiewicz} gives us a lower bound when in $\mathcal{B}$ over $\mathcal{A}$ the number of points is asymptotically smaller than the square of the number of blocks, and \ref{corollary2_zarankiewicz} gives us a lower bound when the number points is asymptotically larger than the square of the number of blocks. This means that for any $\mathcal{B}$ over $\mathcal{A}$ we have $\alpha(\Gamma_{\mathcal{B}}^{m, \lambda}) = \Omega(n^{\frac{m\varepsilon}{m\varepsilon + m - \varepsilon}})$ when choosing 
$\varepsilon = \frac{1}{2}$ which gives us the following theorem.

    \begin{theorem}
        Let $\mathcal{B}$ be a weakly $m$-wise balanced design over $\mathcal{A}$ with parameter $\lambda$ and $n$ incidence pairs. Then $\alpha(\Gamma_{\mathcal{B}}^{m, \lambda}) = \Omega(n^{\frac{m}{3m - 1}}) = \omega(n^{\frac{1}{3}})$.
    \end{theorem}

We also notice that if there is such a weakly $m$-wise balanced design $\mathcal{B}$ over $\mathcal{A}$ for which this lower bound is achieved, then it must be true that $|\mathcal{A}| = \omega(|\mathcal{B}|^{2-\varepsilon})$ and $|\mathcal{A}| = o(|\mathcal{B}|^{2+\varepsilon})$ for every $\varepsilon > 0$.

\subsection{Asymptotic lower bounds for triangle-free graphs}

Let's focus on the case when graphs $\Gamma_{\mathcal{B}}^{m, \lambda}$ are triangle-free. In the case when $m = 1$ and $\lambda = 2$, we have $|\mathcal{A}| \geq \frac{n}{2}$ where $n$ is the number of incidence pairs or vertices in the graph $\Gamma_{\mathcal{B}}^{1, 2}$. 
This is why \ref{upper_bound} doesn't yield any non-trivial asymptotic upper bound. So let's take a look at the case when $m = 2$ and $\lambda = 1$. We denote graphs $\Gamma_{\mathcal{B}}^{m, 1}$ as $\Gamma_{\mathcal{B}}^m$ and graphs $\Gamma_{\mathcal{B}}^{2, 1}$ as $\Gamma_{\mathcal{B}}$. We will derive an inequality which connects the number of points, blocks and incidence pairs in a weakly $m$-wise balanced design with parameter $\lambda = 1$
which will give us tighter lower bounds on the independence number of the corresponding graphs than the bounds obtained in \ref{corollary1_zarankiewicz} and \ref{corollary2_zarankiewicz}.

    \begin{theorem}
        \label{lower_bound_triangle_free}
        Let $\mathcal{B}$ be a weakly $m$-wise balanced design over $\mathcal{A}$ with parameter $\lambda = 1$. Let $n$ be the number of incidence pairs. Let $a = |\mathcal{A}|$ and $b = |\mathcal{B}|$. Then the following inequality holds:
        $$a \geq \frac{n^2}{(m - 1)(b^2 - b) + n}.$$
    \end{theorem}

    \begin{proof}
        For every $p \in \mathcal{A}$ denote with $l_p$ the number of blocks $p$ is contained in. Because any $m$ distinct points are contained together in at most one block, any two blocks intersect in at most $m - 1$ points. This is why the expression $\sum_{p \in \mathcal{A}}\binom{l_p}{2}$ counts every pair of blocks at most $m - 1$ times and so we have $\sum_{p \in \mathcal{A}}\binom{l_p}{2} \leq (m - 1)\binom{b}{2}$ or equivalently $(m-1)(b^2 - b) \geq \sum_{p \in \mathcal{A}}l_p^2 - \sum_{p \in \mathcal{A}}l_p$. 
        By using the Cauchy-Schwarz inequality we get 
        $$(m - 1)(b^2 - b) \geq \frac{1}{a}(\sum_{p \in \mathcal{A}}l_p)^2 - \sum_{p \in \mathcal{A}}l_p.$$
        We know that $n = \sum_{p \in \mathcal{A}}l_p$ and so  
        $$a \geq \frac{n^2}{(m - 1)(b^2 - b) + n}.$$
    \end{proof}

This inequality gives us an asymptotic lower bound for the number of points and blocks in a weakly $m$-wise balanced design with parameter $\lambda = 1$.

    \begin{corollary}
        \label{corollary1_triangle_free}
        Let $\mathcal{B}$ be a weakly $m$-wise balanced design over $\mathcal{A}$ with parameter $\lambda = 1$. Let $\varepsilon > 0$ and $\varepsilon \leq 1$. If $\varepsilon > \frac{1}{2}$, then $|\mathcal{A}|^{\varepsilon} + |\mathcal{B}| = \Omega(n^{\frac{2\varepsilon}{2\varepsilon + 1}})$. Otherwise $|\mathcal{A}|^{\varepsilon} + |\mathcal{B}| = \Omega(n^{\varepsilon})$.
    \end{corollary}

    \begin{proof}
        Let's first prove the case when $\varepsilon > \frac{1}{2}$. If $|\mathcal{B}| = \Omega(n^{\frac{2\varepsilon}{2\varepsilon + 1}})$ the statement is obvious. However, if $|\mathcal{B}| = o(n^{\frac{2\varepsilon}{2\varepsilon + 1}})$, then the inequality from \ref{lower_bound_triangle_free} tells us that $|\mathcal{A}|^{\varepsilon} = \Omega(n^{\frac{2\varepsilon}{2\varepsilon + 1}})$ from which the statement follows.
        Let's now prove the case when $\varepsilon \leq \frac{1}{2}$. If $|\mathcal{B}| = \Omega(n^{\frac{1}{2}})$ the statement is true because $\varepsilon \leq \frac{1}{2}$. However, if $|\mathcal{B}| = o(n^{\frac{1}{2}})$, then \ref{lower_bound_triangle_free} again tells us that $|\mathcal{A}|^{\varepsilon} = \Omega(n^{\varepsilon})$ from which the statement follows.
    \end{proof}

This gives us non-trivial asymptotic lower bounds on the independence number for large subfamilies of graphs $\Gamma_{\mathcal{B}}^m$. The first lower bound gives us information when there are asymptotically less points than the square of the number of blocks in $\mathcal{B}$ over $\mathcal{A}$.

    \begin{corollary}
        \label{corollary2_triangle_free}
        Let $\mathcal{B}$ be a weakly $m$-wise balanced design over $\mathcal{A}$ with parameter $\lambda = 1$. Let $|\mathcal{A}|^{\varepsilon} \leq k|\mathcal{B}|$ for some $k > 0$ and some $\frac{1}{2} < \varepsilon \leq 1$, i. e. $|\mathcal{A}|^{\varepsilon} = O(|\mathcal{B}|)$. Let $n$ be the number of incidence pairs, i. e. the number of vertices in $\Gamma_{\mathcal{B}}^m$. Then $\alpha(\Gamma_{\mathcal{B}}^m) = \Omega(n^{\frac{2\varepsilon}{2\varepsilon + 1}})$.
    \end{corollary}

    \begin{proof}
        We know by \ref{corollary_lower_bound1} that in this case, we have $\frac{|\mathcal{A}|^{\varepsilon} + |\mathcal{B}|}{k+1} \leq \alpha(\Gamma_{\mathcal{B}}^m)$ and so $\alpha(\Gamma_{\mathcal{B}}^m) = \Omega(|\mathcal{A}|^{\varepsilon} + |\mathcal{B}|)$. Corollary \ref{corollary1_triangle_free} tells us that $\alpha(\Gamma_{\mathcal{B}}^m) = \Omega(n^{\frac{2\varepsilon}{2\varepsilon + 1}})$.
    \end{proof}

The second lower bound gives us information about the independence number when there are asymptotically more points than the square of the number of blocks in $\mathcal{B}$ over $\mathcal{A}$.

    \begin{corollary}
        \label{corollary3_triangle_free}
        Let $\mathcal{B}$ be a weakly $m$-wise balanced design over $\mathcal{A}$ with parameter $\lambda = 1$. Suppose $|\mathcal{B}| \leq k|\mathcal{A}|^{1-\varepsilon}$ for some $k > 0$ and $\frac{1}{2} < \varepsilon < 1$, i. e. $|\mathcal{B}| = O(|\mathcal{A}|^{1-\varepsilon})$. Let $n$ be the number of incidence pairs, i. e. the number of vertices in $\Gamma_{\mathcal{B}}^m$. Then $\alpha(\Gamma_{\mathcal{B}}^m) = \Omega(n^{\frac{2\varepsilon}{2\varepsilon + 1}})$.
    \end{corollary}

    \begin{proof}
        We know by \ref{corollary_lower_bound2} that in this case, we have $\frac{|\mathcal{A}|^{\varepsilon} + k|\mathcal{B}|}{2k} \leq \alpha(\Gamma_{\mathcal{B}}^m)$ and so $\alpha(\Gamma_{\mathcal{B}}^m) = \Omega(|\mathcal{A}|^{\varepsilon} + |\mathcal{B}|)$. Corollary \ref{corollary1_triangle_free} then tells us that $\alpha(\Gamma_{\mathcal{B}}^m) = \Omega(n^{\frac{2\varepsilon}{2\varepsilon + 1}})$.
    \end{proof}

Focusing now on the case of triangle-free graphs, i. e. when $m = 2$, we stated the corollaries \ref{corollary2_triangle_free} and \ref{corollary3_triangle_free} only for $\varepsilon > \frac{1}{2}$ because for smaller $\varepsilon$ the theorems from this section would only give us $\alpha(\Gamma_{\mathcal{B}}) = \Omega(n^{\varepsilon})$. But this wouldn't be a new lower bound because the fact that $R(3, t) = \Theta\left(\frac{t^2}{\log{t}}\right)$ tells us that $\alpha(\Gamma_{\mathcal{B}}) = \Omega(\sqrt{n\log{n}})$.

Corollaries \ref{corollary2_triangle_free} and \ref{corollary3_triangle_free} tell us that if there is such a weakly $2$-wise balanced design with parameter $\lambda = 1$ for which this lower bound on $\alpha(\Gamma_{\mathcal{B}})$ is achieved, it must be true that $|\mathcal{A}| = \omega(|\mathcal{B}|^{2-\varepsilon})$ and $|\mathcal{A}| = o(|\mathcal{B}|^{2+\varepsilon})$ for every $\varepsilon > 0$. The corollaries also tell us that if we have $|\mathcal{A}| \leq k|\mathcal{B}|$ for some $k > 0$, the independence number 
of $\Gamma_{\mathcal{B}}$ with $n$ vertices will always be $\Omega(n^{\frac{2}{3}})$ for any choice of $\mathcal{B}$. Let's show that this asymptotic lower bound is tight.

\section{Explicit constructions of Ramsey graphs}

Let's take a look at a special case of the graph $\Gamma_{\mathcal{B}}$ which was already presented in \cite{CodenottiPR00} and \cite{Kostochka2010} where the properties of this explicit construction have also already been proven. If $k$ is a power of some prime number, we construct $P_k$ in such way that we take a finite projective plane of order $k$, that is finite projective plane with $k^2 + k + 1$ points. For vertices we take all incidence pairs $(p, L)$. We take any linear ordering on the set of points and $(p, L)$ and $(p', L')$ are adjacent if $p < p'$, $L \neq L'$ and $p \in L'$. 
Graph $P_k$ is a special case of $\Gamma_{\mathcal{B}}$ in which $\mathcal{A}$ contains $k^2 + k + 1$ points, $\mathcal{B}$ contains $k^2 + k + 1$ blocks and every block contains $k + 1$ points. Any two points are contained together in exactly one block. Results from the previous sections give us the following theorem.

    \begin{theorem}
        $P_k$ doesn't contain any triangles and $\alpha(P_k) \leq 2 \cdot (k^2 + k + 1)$.
    \end{theorem}

Because $P_k$ has $(k^2 + k + 1)(k + 1)$ vertices, we have $\alpha(P_k) = O(n^{\frac{2}{3}})$ where $n = |V(P_k)|$ and projective plane has the same number of points and lines, so the asymptotic lower bound from previous section for cases when $|\mathcal{A}| \leq |\mathcal{B}|$ is tight.

In the same manner, if $k$ is a power of some prime, we could build a graph $A_k$ from a finite affine geometry of order $k$, that is finite affine geometry with $k^2$ points. Such a plane has $k^2 + k$ lines and every line contains $k$ points. Thus the following theorem holds.

    \begin{theorem}
        $A_k$ doesn't contain any triangles and $\alpha(A_k) \leq 2 k^2 + k$.
    \end{theorem}

Because $A_k$ has $k(k^2 + k)$ vertices, we have $\alpha(A_k) = O(n^{\frac{2}{3}})$ where $n = |V(A_k)|$. 

We have given two different constructions of graphs on $n$ vertices with independence number of order $O(n^{\frac{2}{3}})$ but both these constructions are possible only for certain $n$. Let's now construct a graph that has $n$ vertices and independence number of order $O(n^{\frac{2}{3}})$ where $n$ is any natural number. By differentiation we see that $f(k) = \frac{(k+1)^3 + (k+1)^2}{k^3+k^2}$ is decreasing on $[1,\infty)$ and so
$\frac{(k+1)^3 + (k+1)^2}{k^3+k^2} \leq 6$ for $k \in \mathbb{N}$. Thus we have $k$ such that $n \leq k^3 + k^2 \leq 6n$. By Bertrand's postulate \cite{Tchebichef1852} there is a prime number $p$ such that $k \leq p \leq 2k$. Let $\mathcal{A}$ and $\mathcal{B}$ be sets of points and lines in a finite affine geometry of order $p$. Then $\sum_{L \in B}|L| = p^3 + p^2 \geq n$. For every $L \in B$ we can choose 
$L' \subseteq L$ so that $\sum_{L \in B}|L \setminus L'| = n$. Let's now define $\mathcal{B}' = \{L \setminus L'; L \in \mathcal{B}, L \neq L'\}$ and $\mathcal{A}' \subseteq \mathcal{A}$ set of points which are contained in some block in $\mathcal{B}'$. Then $\mathcal{B}'$ is a weakly $2$-wise balanced design over $\mathcal{A}'$ with $n$ incidence pairs. Let's define $G_n = \Gamma_{\mathcal{B'}}$. Then we have $\alpha(G_n) \leq |\mathcal{A}'| + |\mathcal{B}'| \leq 
|\mathcal{A}| + |\mathcal{B}| = 2p^2 + p \leq 8k^2 + 2k \leq 8 \cdot 6\frac{n}{k}$. Because $n \leq k^3 + k^2 \leq 2k^3$ we get $\alpha(G_n) \leq 48\sqrt[3]{2}n^{\frac{2}{3}}$ and thus the following theorem.

    \begin{theorem}
        For every $n \in \mathbb{N}$ graph $G_n$ has $n$ vertices, is triangle-free and $\alpha(G_n) = O(n^{\frac{2}{3}})$.
    \end{theorem}

We should point out that the estimations we made in the previous paragraph are very crude. For every $k \geq 4$ we have $\frac{(k+1)^3 + (k+1)^2}{k^3+k^2} \leq 2$ which is much tighter. When estimating the distribution of prime numbers we only used Betrand's postulate but prime numbers are much denser. Furthermore, we didn't take into account that instead of a prime $p$, we could have taken any power of a prime and the proof would still hold. And
those are even denser. High constant in the estimation could thus be significantly lower.

Finally, let's take a look at a family of graphs we can get by using weakly $2$-wise balanced designs induced by points and lines in the Euclidean plane. Let $N \in \mathbb{N}$ and denote 
\begin{align*}
    \mathcal{P} &= \{(a, b) \in \mathbb{Z}^2; 1 \leq a \leq N, 1 \leq b \leq 2N^2\}\\
    \mathcal{L} &= \{\{(x, mx + b); x \in \mathbb{N}, 1 \leq x \leq N\}; m, b \in \mathbb{N}, 1 \leq m \leq N, 1 \leq b \leq N^2\}.
\end{align*}
Then $|\mathcal{P}| = 2N^3, |\mathcal{L}| = N^3$ and every block contains $N$ points. Let $E_N = \Gamma_{\mathcal{L}}$. Then $E_N$ has $N^4$ vertices and thus $\alpha(E_N) = O(n^{\frac{3}{4}})$ where $n$ is the number of vertices of the graph. Because we know we have an independent set of size $|\mathcal{L}| = n^{\frac{3}{4}}$, we also know that this bound is tight. For any set of points $\mathcal{A}$ and set of lines $\mathcal{B}$ in the Euclidean plane for which $|\mathcal{A}| \leq k|\mathcal{B}|$ for some $k > 0$, we have shown that the independence number
of $\Gamma_{\mathcal{B}}$ will be asymptotically equal to $|\mathcal{A}| + |\mathcal{B}|$. By the Szemerédi–Trotter theorem \cite{Szemerédi1983} we then know that for such families of points and lines in the Euclidean plane, the previous construction has the asymptotically smallest possible independence number, as the theorem states that for any set of points and lines $\mathcal{A}$ and $\mathcal{B}$ in the Euclidean plane with $n$ incidence pairs, we have $|\mathcal{A}| + |\mathcal{B}| = \Omega(n^{\frac{3}{4}})$.

    \begin{theorem}
        We have $\alpha(E_n) = \Theta(n^{\frac{3}{4}})$ where $n = |V(E_N)|$.
    \end{theorem}

At the end let's take a short look at constructions of graphs without $5$-cliques. It's important to point out that in \cite{Kostochka2010} construction of graphs without $5$-cliques was performed using a family of graphs isomorphic to $\Gamma_{\mathcal{B}}^{3, 2}$ for a certian weakly $3$-wise balanced design with parameter $\lambda = 2$. Specifically, they took $\mathcal{A} = \mathbb{Z}_p^3$ where $p$ is some prime and for $\mathcal{B}$ they took a set indexed by $\mathbb{Z}_p^3$ with incidence relation being defined by $(v_1, v_2, v_3) \in B_{(u_1, u_2, u_3)}$
if and only if $(u_1 - v_1)^2 + (u_2 - v_2)^2 + (u_3 - v_3)^2 \equiv 0\; (\text{mod }p)$. By using results from this paper we could then quickly show that $\alpha(\Gamma_{\mathcal{B}}^{3, 2}) = O(n^{\frac{3}{5}})$ where $n$ is the number of vertices which gives us the best currently known constructive lower bound $R(5, t) = \Omega(t^{\frac{5}{3}})$.

\section{Discussion}
In this article we have proven that for any weakly $2$-wise balanced design $\mathcal{B}$ with parameter $\lambda = 1$ that contains $O(|\mathcal{B}|)$ points, the graph $\Gamma_{\mathcal{B}}$ has an independence number of order $\Omega(n^{\frac{2}{3}})$ where $n$ is the number of its vertices. So using such a weakly $2$-wise balanced design can't get us a better constructive asymptotic lower bound on $R(3, t)$ than that determined in \cite{CodenottiPR00} and \cite{Alon94}. However, it remains an open question whether we can get a family of triangle-free graphs with smaller independence number if we use some weakly $2$-wise balanced design with parameter $\lambda = 1$ which has asymptotically strictly more points than blocks. 
Another question left open in this article is whether we can get tighter constructive asymptotic lower bounds on $R(s, t)$ than those determined in \cite{Kostochka2010} and \cite{Alon2001} by providing upper bounds on the independence number of $\Gamma_{\mathcal{B}}^{t,s-t}$ for some choice of weakly $t$-wise balanced designs with parameter $s - t$. To answer these questions, we must determine either a tighter lower bound on the independence number of all graphs $\Gamma_{\mathcal{B}}^{m, \lambda}$ or a tighter upper bound on the independence number of some specific subfamily of these graphs. We have argued in this article that graphs $\Gamma_{\mathcal{B}}^{m, \lambda}$ most likely
to have the smallest independence number are those yielded by some weakly $m$-wise balanced design $\mathcal{B}$ over $\mathcal{A}$ for which $|\mathcal{A}| = \omega(|\mathcal{B}|^{2-\varepsilon})$ and $|\mathcal{A}| = o(|\mathcal{B}|^{2+\varepsilon})$ for every $\varepsilon > 0$. We have also argued that a sensible approach to acquiring these tighter upper bounds on the independence number would be by determining a tighter upper bound on the power of the set $H_2$ from the proof of theorem \ref{upper_bound}.
\printbibliography
\end{document}